\documentclass[reqno]{amsart}
\usepackage{amssymb,amsmath,latexsym,amsthm}
\usepackage{enumitem}
\usepackage[english]{babel}
\usepackage{xcolor}

\usepackage{amsmath,amssymb,amsbsy,amsfonts,amsthm,latexsym,
                       amsopn,amstext,amsxtra,euscript,amscd}

\usepackage{array}
\usepackage{ifthen}
\usepackage{url}
\usepackage{hyperref}

\newtheorem{lem}{Lemma}
\newtheorem{lemma}[lem]{Lemma}

\newtheorem{prop}{Proposition}
\newtheorem{proposition}[prop]{Proposition}

\newtheorem{thm}{Theorem}
\newtheorem{theorem}[thm]{Theorem}

\def\\{\cr}
\def\({\left(}
\def\){\right)}
\def\[{\left[}
\def\]{\right]}
\def\<{\langle}
\def\>{\rangle}

\allowdisplaybreaks

\begin{document}

\title{On the distance between factorials and repunits}

\author{Michael Filaseta}
\address{Mathematics Department, University of South Carolina, Columbia, SC 29208, USA}
\email{filaseta@math.sc.edu}

\author{Florian Luca}
\address{Mathematics Division, Stellenbosch University, Stellenbosch, South Africa}
\email{fluca@sun.ac.za}

\date{\today}

\pagenumbering{arabic}

\maketitle

\begin{abstract}
We show that if $n\ge n_0,~b\ge 2$ are integers, $p\ge 7$ is prime and  $n!-(b^p-1)/(b-1)\ge 0$, then $n!-(b^p-1)/(b-1) \ge 0.5\log\log n/\log\log\log n$.  Further results are obtained, in particular for the case $n!-(b^p-1)/(b-1) < 0$.
\end{abstract}

\bigskip

\centerline{\parbox[h]{10cm}{{\footnotesize \bf 2020 AMS Subject Classification:} {\footnotesize 11D61}.\\[5pt]
}}
\bigskip

\centerline{\it Dedicated to Carl Pomerance on the occasion of his $80^{\rm th}$ birthday.}

\bigskip

\section{Introduction}
There are many Diophantine equations involving factorials which were studied in the literature. A famous one formulated by H.~Brocard in 1976 (see \cite{Br1} and \cite{Br2}) and again by S.~Ramanujan \cite[Question 469, p.~327]{Ram} in 1913 asks for all the positive integer solutions  $(n,x)$ of $n!=x^2-1$. The only known solutions are $(n,x) \in \{(4,5), (5,11), (7,71)\}$. B.~Berndt and W.~Galway \cite{BG} proved in 2000 that there are no other solutions in the range $n\in [1,10^9]$. P.~Erd\H os and R.~Obl\'ath \cite{EO} treated the Diophantine equations $n!=x^k\pm y^k$ as well as $m!\pm n!=x^k$. A related problem concerning the positive integer solutions  $(n,x)$ of the Diophantine equation $n!=x(x+1)$ was posed by P.~Erd\H os at one of the Western Number Theory Conferences (see \cite{BH}). 

Here we look at Diophantine equations involving factorials and repunits. 
We use $b$ to represent a base, so for the rest of the paper, we will have $b \ge 2$ and only mention this condition on $b$ if we want to remind the reader of it.  
Recall that a repunit in base $b$ is an integer of the form $(b^m-1)/(b-1)$, where $m \in \mathbb Z^{+}$.  It is called a repunit in base $b$ since its base $b$ representation consists of a string of $1$'s. The exponent $m$ is the length of the repunit, or the number of its digits. Rather than asking for a factorial to be a repdigit in base $b$ and solving the corresponding Diophantine equation, we ask the somewhat more general question whether a factorial and a repunit can be close. In other words, we look at the difference
\begin{equation}
\label{eq:x0}
n!-\frac{b^m-1}{b-1}=a.
\end{equation}
The case $m = 1$ is not interesting.  
The case $m = 2$ is not much better since for every choice of $n$ and $a$ with $n! - a \ge 3$, the equation \eqref{eq:x0} with $m = 2$ has a solution in $b \ge 2$. 
We therefore restrict to $m \ge 3$.  
When $a=0$, we get that 
$$
n!=\frac{b^m-1}{b-1}.
$$
The main result in \cite{LLS} shows that $m\in \{3,4,6,12\}$. 
We will focus on the case that $m = p$, a prime, and the only prime in $\{3,4,6,12\}$ is $m=3$.  For this choice of $m$, the equation becomes 
$$
n!=b^2+b+1.
$$
Since the number on the right is never divisible by $9$, it follows that $n\le 5$, and one checks that there is no solution $(n,b)$ with $n\in \{1,2,3,4,5\}$ and $b \ge 2$.  Thus, writing
\begin{equation}
\label{eq:x}
n! = (b^p-1)/(b-1) + a,
\end{equation}
we conclude that $a \ne 0$.  Our main theorem is that $a \in \mathbb Z^{+}$ cannot be small.

\begin{theorem}
\label{thm:1}
There exists $n_0$ such that equation \eqref{eq:x} with integers $n\ge n_0,~b\ge 2$, prime $p\ge 7$ and $a \in \mathbb Z^{+}$ implies 
\begin{equation}
\label{thm1eq}
a \ge \frac{0.5\log\log n}{\log\log\log n}.
\end{equation}
\end{theorem}

Observe that for each positive integer $n$, there are at most finitely many solutions to \eqref{eq:x} in integers $b\ge 2$, primes $p\ge 7$ and $a \in \mathbb Z^{+}$.   Hence, an alternative statement of the result is that equation \eqref{eq:x} has finitely many solutions in integers $n\ge 1,~b\ge 2$, primes $p\ge 7$ and positive integers $a < 0.5\log\log n/\log\log\log n$.

It is of some interest to address what happens when $a$ is negative in \eqref{eq:x}.  The following result addresses this case, where the notation $\Phi_{p}(X)$ refers to the $p^{\rm th}$ cyclotomic polynomial.

\begin{theorem}
\label{thm:2}
Let $b\ge 2$, $n \ge 2$ and $a$ be integers, and let $p \ge 7$ a prime.  Then
\begin{enumerate}[leftmargin=1cm,itemsep=5pt]
\item[(i)]
if there are solutions in $n$ to \eqref{eq:x} with $|a| \le n$, then $p \mid a$ or $|a| \equiv 1 \pmod{p}$;
\item[(ii)]
if $\Phi_{p}(X) + a$ is irreducible or a linear factor times an irreducible polynomial and if $n$ is sufficiently large  and satisfying \eqref{eq:x}, then 
\begin{equation}
\label{thm2eq}
|a| \ge \frac{0.5\log\log\log n}{\log\log\log\log n};
\end{equation}
\item[(iii)]
if $n$ is sufficiently large, $a < 0$, $p \mid a$ and \eqref{eq:x} holds, then \eqref{thm2eq} holds;
\item[(iv)]
if $2 \le |a| \le n$ and \eqref{eq:x} holds, then
$a$ is odd and every prime factor $q \ne p$ of $a$ satisfies $q \equiv 1 \pmod{p}$;
\item[(v)]
if $2 \le |a| \le \min\{ 100000, n \}$ and
\begin{align*}
(a,p) \not\in S_{0}:= &\bigg(  \bigcup_{7 \le p \le 100000} \{ (-p,p) \}  \bigg) \bigcup \{ 
(-43,7), (-127,7), (-547,7), (-683, 11), \\
&\qquad\quad (-1093, 7), (-2047,11), (-2731,13), (-3277,7), (-5461,7), \\
&\qquad\quad (-8191, 11), (-13021, 7), (-19531, 7), (-39991,7),  \\
&\qquad\quad (-43691, 17), (-44287, 11), (-55987, 7), (-88573, 11)
 \},
\end{align*}
then \eqref{eq:x} does not hold;
\item[(vi)]
if $n$ is sufficiently large, $(a,p) \in S_{0}$ and \eqref{eq:x} holds, then \eqref{thm2eq} holds.
\end{enumerate}
\end{theorem}

Observe that (v) is indicating that in Theorem~\ref{thm:1}, except for the choices of $a$ arising in $S_{0}$, one can replace \eqref{thm1eq} with $a > n$ whenever $2 \le a \le 10^{5}$.  This is of some significance given that the right-hand side of \eqref{thm1eq} is less than $10^{5}$ for $n \le 10^{10^{1000000}}$.  So the value of Theorem~\ref{thm:1} is not in what happens for a particular value of $n$ in \eqref{eq:x} but rather in that it provides a bound on $a$ that universally holds for all $n$.  It is also perhaps noteworthy that the only instance of $(a,p)$ in $S_{0}$ where $p \mid a$ besides the pairs $(-p,p)$ is $(-39991,7)$.

Now we make some comments before passing to the proofs. The Diophantine equation \eqref{eq:x0} can be rewritten as 
\begin{equation}
\label{eq:1}
n!=\frac{b^m-1}{b-1}+a.
\end{equation}
The right--hand side is a polynomial $P_{m,a}(X)=(X^m-1)/(X-1)+a$ in $X=b$  for fixed $a$ and $m$. The existence of a solution $(n,b)$ to \eqref{eq:1} with $b\ge 2$ and $n$ large means that 
the congruence $P_{m,a}(X)\equiv 0\pmod q$ has a simultaneous solution for all primes $q \le n$, and in fact even for all composite moduli $q \le n$. Conditionally upon the $abc$ conjecture, the second author \cite{Lu} showed that for every fixed polynomial $Q(X)\in {\mathbb Z}[X]$ of degree at least $2$, the Diophantine equation $n!=Q(b)$ has only finitely many integer solutions $(n,b)$.  His results can be easily adapted to show that the $abc$ conjecture implies
the Diophantine equation \eqref{eq:1} has only finitely many solutions $(n,m,b,a)$ with $|a|=O(n)$, $b\ge 2$ and $m\ge 3$. Unconditionally, D.~Berend and J.~E.~Harmse \cite{BH} showed the equation $n!=Q(b)$ 
has only finitely many integer solutions $(n,b)$ under certain assumptions  on $Q(X)$. These assumptions do not hold for the polynomial $P_{5,-1}(X)=X(X+1)(X^2+1)$, so one cannot deduce in particular that equation 
\eqref{eq:1} has only finitely many solutions $(n,b)$ when $a=-1$ and $m=5$, for example. It would certainly be interesting to prove some result in the spirit of Theorem~\ref{thm:1} by removing the restriction that the parameter $p$ is prime. We leave this question as a project for the interested reader.

\section{Lemmas and proofs}

To begin with, we are interested in the solutions to equation \eqref{eq:1}
in integers $n \ge 2$,~$a \ne 0$,~$b\ge 2$ and $m\ge 3$.  

 \begin{lemma}
 \label{lem:1}
 Assume that $|a|=O(n)$ in equation \eqref{eq:1}.  Then 
 $$
 m \ll n^{1/2}( \log n)^2
 \qquad \text{and} \qquad
b \ge \exp\big( c_{1} \sqrt{n}/\log n \big),
 $$
 for some absolute constant $c_{1} > 0$.
 \end{lemma}

The implied constants here and elsewhere, unless stated otherwise, are absolute.  In this lemma, we may take $n$ to be large since for bounded $n$ we have, by \eqref{eq:1}, that $m$ is also bounded and the conclusion $m\ll n^{1/2}( \log n)^2$ follows.  
The conclusion also remains valid in a wider range for the parameter $a$ (for example, if $a=n^{O(1)}$), but we will only be interested in $a$ in a much smaller range.

\begin{proof}
Let $n$ be large.  
The left--hand side in \eqref{eq:1} is $n! \asymp (n/e)^{n+O(1)}$ while the right--hand side is 
$\asymp \max \{b^{m-1}, |a|\}$. Taking logarithms on both sides of \eqref{eq:1} and recalling $|a|=O(n)$, we get that 
$$
m\log b\asymp n\log n.
$$
We next rewrite \eqref{eq:1} as 
\begin{equation}
\label{eq:lem1eq1}
(b-1)n!=b^{m} + (a(b-1)-1).
\end{equation}
If $a(b-1)-1=0$, then $a=1,~b=2$, and we get the equation $n!=2^m$ which has no solutions with $m\ge 2$. 
If $a(b-1)-1=\pm 1$, then $a \ne 0$ and $b \ge 2$ imply $b\in \{2,3\}$ and either 
$n! = 2^m + 1$ or $2 \cdot n! = 3^{m} + 1$. 
As $2^m + 1$ is odd, the equation $n! = 2^m + 1$ does not hold since $m \ge 3$.
Since $3^m + 1$ is not divisible by $3$, the equation $2 \cdot n! = 3^m + 1$ does not hold for $n \ge 3$.
As we are taking $n$ large, we deduce $|a(b-1) - 1|\ge 2$.  

Next, we compare the $2$-adic valuations of both sides of \eqref{eq:lem1eq1}. 
In the left--hand side, we have
\begin{equation}
\label{eq:lem1eq2}
\nu_2\big((b-1)n!\big) \ge \nu_2(n!)=\bigg\lfloor \frac{n}{2} \bigg\rfloor + \bigg\lfloor\frac{n}{4} \bigg\rfloor +\cdots > \frac{n}{2}
\end{equation}
for $n\ge 4$. 
In the right--hand side, we have, using a linear form in $2$-adic logarithms (cf.~\cite[Theorem~2.9]{YBug}), 
\begin{align*}
\nu_2(b^m + ((a(b-1)-1))) &\ll \log b \log |a(b-1)-1| \log m \\
&\ll \log b (\log b+\log n) \log m.
\end{align*}
We thus get that
$$
n\ll \log b (\log b+\log n)\log m\ll (\log b)^2 \log n \log m.
$$
Hence, 
\begin{equation}
\label{eq:lem1eq3}
\log b\gg \left(\frac{n}{(\log n)(\log m)}\right)^{1/2}.
\end{equation}
We thus get that
$$
n\log n\asymp m\log b\gg m \left(\frac{n}{(\log n)(\log m)}\right)^{1/2},
$$
which gives us that 
$$
n^{1/2} (\log n)^{3/2} \gg \frac{m}{(\log m)^{1/2}}.
$$
This implies $m\ll n^{1/2} (\log n)^2$. 
As a consequence, we have $\log m \ll \log n$, and the lemma now follows from \eqref{eq:lem1eq3}.
\end{proof}

\medskip

Although our main result addresses the case in equation \eqref{eq:1} where $a = -1$ and $m$ is prime, we provide arguments next for the case $a = -1$ and all sufficiently large integers $m$. 

\begin{lemma}
\label{lemmatwo}
There is an $m_{0}$ such that \eqref{eq:1} has only finitely many solutions $(n,a,b,m)$ in integers $n \ge 2$, $a=-1$, $b \ge 2$ and $m \ge m_{0}$. 
\end{lemma}

Note that one solution to \eqref{eq:1} satisfying the conditions $n \ge 2$, $a=-1$ and $b \ge 2$ of Lemma~\ref{lemmatwo} is $(n,a,b,m) = (5,-1,3,5)$.  
We will give two proofs of Lemma~\ref{lemmatwo}.  The first proof will be shorter, but the second will establish a considerably better bound on $m_{0}$, namely $m_{0} = 6$, and will be based on a proposition we establish here of independent interest.  
For the first proof, we give the value $m_{0} = 4.0515 \cdot 10^8$ just to clarify some explicit value of $m_{0}$ that can be obtained from the method; improvements on this value of $m_{0}$ are possible by sharpening, for example, the bounds in \eqref{lem2proof1eq3} and \eqref{lem2proof1eq4} and the choice of $J$ below.  

\begin{proof}[First proof of Lemma~\ref{lemmatwo} (with $m_{0} = 4.0515 \cdot 10^8$)]
To start, we consider $m \ge 3$.  
Let $m' = m-1 \ge 2$.  We rewrite \eqref{eq:1} with $a=-1$ as
\begin{equation}
\label{lem2proof1eq1}
n! = \frac{b^m-1}{b-1} - 1
= b \bigg(  \dfrac{b^{m-1}-1}{b-1}  \bigg)
= b \bigg(  \dfrac{b^{m'}-1}{b-1}  \bigg).
\end{equation}
Let $R$ denote the right--hand side of the last equation in \eqref{lem2proof1eq1}.  
Observe that there are finitely many solutions for each fixed $n \ge 2$ since 
$$
n! = R = b^{m'} + b^{m'-1} + \cdots + b \ge 2^{m'} + 2^{m'-1} + \cdots + 2 = 2^{m'+1} - 2
$$
implies $m'$ (and hence $m$) is bounded and, for each fixed $n$ and $m'$, there is at most one positive integer solution to \eqref{lem2proof1eq1} (by Descartes' Rule of Signs).  Thus, we may suppose that $n$ is large.  So we assume that for some $n$ large that we have a solution to \eqref{lem2proof1eq1} in $b \ge 2$ and $m' \ge 2$ with the goal of obtaining a contradiction.  (Note that by a direct computation, we were able to show that $(n,a,b,m) = (5,-1,3,5)$ is the only solution satisfying the conditions in Lemma~\ref{lemmatwo} with $n \le 100$ and $m_{0} = 3$.)

Denoting the $m'$-th cyclotomic polynomial by $\Phi_{m'}(x)$, we see that $R$ has the factor $\Phi_{m'}(b)$ which is divisible by possibly one prime (to the first power) that divides $m'$ and otherwise only by primes that are $1$ modulo $m'$ (cf.~\cite{Dic}).  
Also, $\Phi_{m'}(b)$ contributes a factor of $R$ (equivalently, of $n!$ by \eqref{lem2proof1eq1}) that is of size 
\begin{align*}
\Phi_{m'}(b) &= \prod_{\substack{1 \le k \le m' \\ \gcd(k,m') = 1}} \bigg(  b - e^{2\pi i k/m'}  \bigg) \\
&\ge \prod_{\substack{1 \le k \le m' \\ \gcd(k,m') = 1}} \bigg(  b - \big| e^{2\pi i k/m'} \big| \bigg) \\[4pt]
&= (b-1)^{\phi(m')}.
\end{align*}
On the other hand, we claim that $R < (b-1)^{m'+1}$.  In fact, this inequality holds provided
$$
b^{m'+1} \le (b-1)^{m'+2} = b^{m'+1} \cdot b \bigg( 1 - \dfrac{1}{b} \bigg)^{m'+2}
$$
which follows provided
$$
b \bigg( 1 - \dfrac{1}{b} \bigg)^{m'+2} \ge 1.
$$
We can show that this is the case with an appropriate upper bound on $m'$ and lower bound on $b$. 
In particular, since $n$ is large, we know from Lemma~\ref{lem:1} that $m'+2 \le n-1$ and $b \ge n$ so that
$$
b \bigg( 1 - \dfrac{1}{b} \bigg)^{m'+2} 
\ge n \bigg( 1 - \dfrac{1}{n} \bigg)^{n-1} > \dfrac{n}{e} > 1,
$$
where we have used that $(1-(1/n))^{n-1} > 1/e$ since $n \ge 1$.  
We deduce now that
\begin{equation}
\label{lem2proof1eq2}
\begin{split}
\Phi_{m'}(b) &\ge (b-1)^{\phi(m')} = \big( (b-1)^{m'+1} \big)^{\phi(m')/(m'+1)} \\ 
&> R^{\phi(m')/(m'+1)} = (n!)^{\phi(m')/(m'+1)}.
\end{split}
\end{equation}

As noted above, $\Phi_{m'}(b)$ has a factor $\ge \Phi_{m'}(b)/m'$ dividing $n!$ which is only divisible by primes that are $1$ modulo $m'$.  The idea now is to establish a contradiction by showing that the product of the primes congruent to $1$ modulo $m'$ dividing $n!$ cannot be large enough for \eqref{lem2proof1eq2} to hold.
We use that
$$
\nu_p(n!)=\bigg\lfloor \frac{n}{p} \bigg\rfloor + \bigg\lfloor\frac{n}{p^{2}} \bigg\rfloor +\cdots < \frac{n}{p-1}.
$$
Thus, the contribution of primes $p \equiv 1 {\hskip -2pt}\pmod{m'}$ to $\log n!$ is at most
$$
\sum_{\substack{p \le n \\ p \equiv 1 {\hskip -6pt}\pmod{m'}}} \dfrac{n \log p}{p-1}
\le \sum_{1 \le k \le (n-1)/m'} \dfrac{n \log n}{km'}
\le \sum_{1 \le k < n/2} \dfrac{n \log n}{km'} \le \dfrac{n \log^{2}n}{m'},
$$
where the last inequality comes from well-known estimates on bounds for partial sums of the harmonic series (cf.~\cite[(9.89)]{GKP}).
On the other hand, by \eqref{lem2proof1eq2} the contribution of primes $p \equiv 1 {\hskip -2pt}\pmod{m'}$ to $\log n!$ is at least
\begin{equation}
\label{lem2proof1eq3}
\log \big((n!)^{\phi(m')/(m'+1)}/m' \big) = \dfrac{\phi(m')}{m'+1} \log n! - \log m'
\ge  \dfrac{\phi(m') n \log n}{2m'},
\end{equation}
where this last inequality follows from $e^{n} = \sum_{j=0}^{\infty} n^{j}/j! \ge n^{n}/n!$, the upper bound on $m$ in Lemma~\ref{lem:1}, $m' = m-1 \ge 2$ and $n$ being large.  Comparing the upper and lower bounds on the contribution of primes $p \equiv 1 {\hskip -2pt}\pmod{m'}$ to $\log n!$, we obtain $\phi(m') \le 2 \log n$.  We now deduce that
\begin{equation}
\label{lem2proof1eq4}
m' \le \log^{2} n
\end{equation}
by well-known lower bounds for the $\phi$-function (cf.~\cite[Theorem~328]{HW} or \cite[Theorem~15]{RS}).

To obtain our contradiction, we have more work to do.  
We recompute the upper bound on the contribution of the primes $p \equiv 1 {\hskip -2pt}\pmod{m'}$ to $\log n!$.  Since $(\log x)/(x-1)$ is a decreasing function of $x$ for $x > 1$, we have 
\begin{align*}
\sum_{\substack{p \le n \\ p \equiv 1 {\hskip -6pt}\pmod{m'}}} \dfrac{n \log p}{p-1}
&\le \sum_{1 \le j \le \lfloor \log n/\log 2 \rfloor} 
\sum_{\substack{2^{j} \le p < 2^{j+1} \\ p \equiv 1 {\hskip -6pt}\pmod{m'}}} \dfrac{n \log (2^{j})}{2^{j}-1} \\[5pt]
&= \sum_{1 \le j \le \lfloor \log n/\log 2 \rfloor} \dfrac{n j \log 2}{2^{j}-1}
\sum_{\substack{2^{j} \le p < 2^{j+1} \\ p \equiv 1 {\hskip -6pt}\pmod{m'}}} 1.
\end{align*}
Set $J = 4 \log\log n$.  
By \eqref{lem2proof1eq4}, for $j \ge J$, we have $2^{j} > e^{2 \log\log n} = \log^{2} n \ge m'$.  
We evaluate the inner sum above by applying Montgomery and Vaughan's version of the Brun--Titchmarsh Theorem \cite[Theorem~2]{MV}.  We deduce that 
the contribution of primes $p \equiv 1 {\hskip -2pt}\pmod{m'}$ to $\log n!$ is at most
$$
n (\log 2) \bigg( \sum_{1 \le j < J} \dfrac{j \,2^{j}}{2^{j}-1} + \sum_{J \le j \le \lfloor \log n/\log 2 \rfloor} \dfrac{j}{2^{j}-1} \cdot \dfrac{2^{j+1}}{\phi(m') \log(2^{j}/m')} \bigg).
$$
The first of these last two sums can be estimated by
$$
\sum_{1 \le j < J} \dfrac{j \,2^{j}}{2^{j}-1} \le \sum_{1 \le j < J} (j+1) \le \frac{(J+1)(J+2)}{2}-1<J^{2} = 16 (\log \log n)^{2},
$$
for $n\ge 9$. For the second sum, we use that from \eqref{lem2proof1eq4}, we have $\log m' \le J/2 \le j/2$ so that
$$
\log(2^{j}/m') = j \log 2 - \log m' \ge (\log 2 - 0.5) j \ge 0.19 j.
$$
Thus, we obtain
\begin{align*}
\sum_{J \le j \le \lfloor \log n/\log 2 \rfloor} \dfrac{j}{2^{j}-1} \cdot &\dfrac{2^{j+1}}{\phi(m') \log(2^{j}/m')} \\[5pt]
&\le \sum_{J \le j \le \lfloor \log n/\log 2 \rfloor} \dfrac{2^{j}}{2^{j}-1} \cdot \dfrac{2}{0.19 \,\phi(m')} \\[5pt]
&\le \left(\frac{\log n}{\log 2}-J+1+\sum_{j\ge J }\frac{1}{2^j-1}\right) \dfrac{2}{0.19 \,\phi(m')}\\
&\le \dfrac{2 \,\log n}{0.19\,\phi(m')\,\log 2}
\end{align*}
for $n\ge 9$. Hence, we now see that the contribution of primes $p \equiv 1 {\hskip -2pt}\pmod{m'}$ to $\log n!$ is at most
$$
16\,(\log 2)\,n\,(\log \log n)^{2}
+ \dfrac{2 \,n\,\log n}{0.19\,\phi(m')}.
$$
Let the first and second of these terms above be $T_{1}$ and $T_{2}$, respectively.
Recall that \eqref{lem2proof1eq3} gives a lower bound, say $L$, on this same quantity. 
Since $L \le T_{1} + T_{2}$, we deduce that either $T_{1} \ge (2/3) L$ or $T_{2} \ge (1/3) L$.  
Therefore, either
$$
\phi(m') \le \dfrac{48\,(\log 2)\,(\log \log n)^{2}\,m'}{\log n}
$$
or
$$
\phi(m') \le \sqrt{12 m'/0.19}.
$$
For $n \ge 1619$, the function $(\log\log n)^{2}/\log n$ is decreasing.  
For the next part, we use that $m \ge m_{0} = 56$.  
For such $m$, we obtain from \eqref{lem2proof1eq4} that $n \ge e^{\sqrt{m'}} \ge e^{\sqrt{55}} > 1619$ so that
$$
\dfrac{(\log\log n)^{2}}{\log n} 
\le \dfrac{(\log \sqrt{m'}\,)^{2}}{\sqrt{m'}} = \dfrac{\log^{2} m'}{4 \sqrt{m'}}.
$$
Given that one of the two upper bounds on $\phi(m')$ hold, we now see that
\begin{equation}
\label{lem2proof1eq5}
\phi(m') \le 8.32 \,\sqrt{m'}\,\log^{2} m'.
\end{equation}
Lower bounds on $\phi(m')$ lead to a contradiction.  In particular, from \cite[Theorem~15]{RS}, we have
$$
\phi(m')> \dfrac{m'}{e^{\gamma} \log\log m' + 2.50637/\log\log m'},
$$
where $\gamma = 0.5772156649015\ldots$ is Euler's constant.  We deduce that \eqref{lem2proof1eq5} cannot hold if
$$
m' > \big( 8.32 \,(\log^{2} \!m') (e^{\gamma} \log\log m' + 2.50637/\log\log m') \big)^{2}.
$$
Now, we consider $m \ge m_{0} = 4.0515 \cdot 10^8$.  
We use that $\gamma < 0.577215665$ and $m' = m-1 \ge 4.0515 \cdot 10^8 - 1$, substituting the latter into the expression $\log \log m'$ appearing in the denominator above.
We see that to obtain a contradiction, it suffices to show $f(m') > 0$ where
$$
f(x) = x - 69.2224 \,(\log^{4}x)\,(1.78107242\,\log (\log x) + 0.83918269)^2.
$$
One checks that this is an increasing function of $x$ for $x \ge 4.0515 \cdot 10^8-1$, finishing the proof.
\end{proof}

Before turning to our second proof of Lemma~\ref{lemmatwo}, we give some background.  
Let $\ell$ and $k\ge 2$ be coprime integers. In Theorem 1 and Remark 1 in \cite{Pom}, Pomerance puts
\begin{equation}
\label{eq:x1}
S(x;k,\ell):=\bigg(  \sum_{\substack{p\le x\\ p\equiv \ell{\hskip -6pt}\pmod k}}\frac{1}{p}  \bigg)-\frac{\log\log x}{\phi(k)},
\end{equation}
and proves that 
\begin{equation}
\label{eq:x2}
|S(x;k,\ell)|=\frac{1}{p_{k,\ell}}+O\left(\frac{\log k}{\phi(k)}\right),
\end{equation}
where $p_{k,\ell}$ is the smallest prime number $p\equiv \ell \pmod k$.   The term $1/p_{k,\ell}$ can be omitted if $p_{k,\ell}$ exceeds $k$. 
In our application we need an explicit upper bound in terms of $k$ for the similar but larger sum
\begin{equation}
\label{eq:x3}
T(x;k,\ell): =\bigg(  \sum_{\substack{p\le x\\ p\equiv \ell{\hskip -6pt}\pmod k}}\frac{\log p}{p}  \bigg)-\frac{\log x}{\phi(k)}.
\end{equation}
Using the same notation for $p_{k,\ell}$, here is our result.

\begin{proposition}
\label{prop:Pom}
The estimate 
\begin{equation}
\label{eq:x4}
|T(x;k,\ell)|=\frac{\log p_{k,\ell}}{p_{k,\ell}}+O\left(\frac{{\sqrt{k}}\log k}{\phi(k)}\right) 
\end{equation}
holds uniformly in $x\ge 2$. The ${\sqrt{k}}$ above can be replaced by $k^{\delta}$ for any fixed $\delta>0$ with a constant implied by the $O$-symbol depending on $\delta$. 
\end{proposition}

\begin{proof}
In order to deduce \eqref{eq:x2} Pomerance splits the sum in \eqref{eq:x1} at $e^k$ and applies  the Brun--Titmarsch theorem and partial summation. We follow the same approach except that we split the sum \eqref{eq:x3} at $e^{\sqrt{k}}$. Recall that the Brun--Titchmarsh theorem (cf.~\cite{MV}) is the inequality 
$$
\pi(x;k,\ell)\ll \frac{x}{\phi(k)\log(x/k)}
$$
which holds for all $x>k$. The implied constant can be taken to be $2$. We have
\begin{equation}
\label{eq:x7}
T(x;k,\ell)=T_1+T_2,
\end{equation}
where
$$
T_1:=\sum_{\substack{p\le e^{\sqrt{k}}\\ p\equiv \ell{\hskip -6pt}\pmod k}} \frac{\log p}{p}
\qquad \text{and} \qquad
T_2:=\bigg(  \sum_{\substack{e^{\sqrt{k}} < p \le x\\ p\equiv \ell{\hskip -6pt}\pmod k}} \frac{\log p}{p}  \bigg)-\frac{\log x}{\phi(k)}.
$$
If $x<e^{\sqrt{k}}$, then the sum in $T_2$ is not present while $\log x/\phi(k)\le {\sqrt{k}}/\phi(k)$, which is smaller than the error term in \eqref{eq:x4}.  

For $T_1$ (and general $x \ge 2$), we have
\begin{equation}
\label{eq:T1}
T_1=\delta_1 \frac{\log p_1}{p_1}+\delta_2\frac{\log p_2}{p_2}
+\sum_{\substack{2k<p\le e^{\sqrt{k}}\\ p\equiv \ell {\hskip -6pt}\pmod k}} \frac{\log p}{p},
\end{equation}
where $p_1$ and $p_2$ are the first two primes $p\equiv \ell\pmod k$ and $\delta_i = 0$ if $p_i>2k$ for $i\in \{1,2\}$ and $\delta_i = 1$ otherwise. Certainly, $\log p_1/p_1=\log p_{k,\ell}/p_{k,\ell}$, while if $\delta_2=1$, then $\log p_2/p_2=O((\log k)/k)$ which is smaller than the error term in \eqref{eq:x4}. 
As for the sum in \eqref{eq:T1}, by partial summation, we have 
$$
\sum_{\substack{2k<p\le e^{\sqrt{k}}\\ p\equiv \ell {\hskip -6pt}\pmod k}} \frac{\log p}{p}
=\frac{\pi(y;k,\ell) \log y}{y}\Big|_{y=2k}^{y=e^{\sqrt{k}}}+\int_{2k}^{e^{\sqrt{k}}} \frac{\pi(y;k,\ell)(\log y-1)}{y^2} dy.
$$
By the Brun--Titchmasrh inequality, the first term is 
$$
\ll \frac{y\log y}{y\phi(k)\log(y/k)}\Big|_{y=2k}^{y=e^{\sqrt{k}}}\ll \frac{\log k}{\phi(k)}. 
$$
In the integral, we have
\begin{eqnarray*}
\int_{2k}^{e^{\sqrt{2k}}} \frac{\pi(y;k,\ell)(\log y-1)}{y^2} dy & \ll & \int_{2k}^{e^{\sqrt{k}}} \frac{(\log y)dy}{\phi(k)y\log(y/k)}\\
& \le & \int_{2k}^{e^{\sqrt{k}}} \frac{(\log(y/k)+\log k)dy}{\phi(k)y\log(y/k)}\\ 
& \le &\int_{2k}^{e^{\sqrt{k}}} \frac{dy}{\phi(k)y}+\log k\int_{2k}^{e^{\sqrt{k}}} \frac{dy}{\phi(k)y\log(y/k)} \\ 
& \le & \left(1+\frac{\log k}{\log 2}\right) \int_{2k}^{e^{\sqrt{k}}} \frac{dy}{\phi(k) y} \\ 
& = & \left(1+\frac{\log k}{\log 2}\right) \frac{1}{\phi(k)} \log y\Big|_{y=2k}^{y=e^{\sqrt{k}}}\\
& \ll & \frac{{\sqrt{k}} \log k}{\phi(k)}.
\end{eqnarray*}
Thus, the sum in \eqref{eq:T1} sits entirely in the error term in \eqref{eq:x4}. 

For $T_2$, we consider $y \in [e^{\sqrt{k}},x]$.  In particular, we have $k \le (\log y)^2$, so we are in the range of the Siegel--Walfisz theorem (cf.~\cite[Lemma~2.9]{Ell}) giving
$$
\pi(y;k,\ell)=\frac{\pi(y)}{\phi(k)}+O\left(\frac{y}{\exp(A(\log y)^{1/2})}\right)
$$
for some absolute constant $A > 0$. We use again the Abel summation formula to obtain
\begin{equation}
\label{eq:x6}
\sum_{\substack{e^{\sqrt{k}} <  p\le x\\ p\equiv \ell{\hskip -6pt}\pmod k}} \frac{\log p}{p}=\frac{\pi(y;k,\ell)\log y}{y}\Big|_{y=e^{\sqrt{k}}}^{x}+\int_{e^{\sqrt{k}}}^x \frac{\pi(y;k,\ell)(\log y-1)}{y^2} dy.
\end{equation}
By the Brun--Titchmarsh theorem, the first term is $O(1/\phi(k))$.  
For the integral, we have 
\begin{equation}
\label{eq:T2}
\begin{split}
\int_{e^{\sqrt{k}}}^x \frac{\pi(y;k,\ell)(\log y-1)}{y^2} dy & = \frac{1}{\phi(k)}\int_{e^{\sqrt{k}}}^x \frac{\pi(y)(\log y-1)}{y^2} dy \\[5pt]
&\qquad \qquad +  O\left(\int_{e^{\sqrt{k}}}^{x} \frac{\log y\,dy}{y\exp(A(\log y)^{1/2})}\right).
\end{split}
\end{equation}
The integrand from the error term is $\ll 1/\big(y(\log y)\big)^2$ so that
$$
\int_{e^{\sqrt{k}}}^{x} \frac{\log y\,dy}{y\exp(A(\log y)^{1/2})}
\ll \int_{e^{\sqrt{k}}}^{\infty} \frac{dy}{y (\log y)^2}
=-\frac{1}{\log y}\Big|_{y=e^{\sqrt{k}}}^{\infty}
=\frac{1}{\sqrt{k}}
\le \frac{\sqrt{k}}{\phi(k)}.
$$
For the first integral on the right--hand side of \eqref{eq:T2}, using the Prime Number Theorem in the form
$$
\pi(y) = \dfrac{y}{\log y} + \dfrac{y}{\log^{2} y} + O\bigg(  \dfrac{y}{\log^{3} y}  \bigg),
$$
we obtain
\begin{align*}
\int_{e^{\sqrt{k}}}^x \frac{\pi(y)(\log y-1)}{y^2} dy 
&= \int_{e^{\sqrt{k}}}^x \frac{\log y - 1}{y \log y} + \frac{\log y - 1}{y \log^{2} y} dy +
O\bigg(  \int_{e^{\sqrt{k}}}^x \frac{\log y - 1}{y \log^{3} y} dy   \bigg) \\[5pt]
&= \int_{e^{\sqrt{k}}}^x \frac{dy}{y} + O\bigg(  \int_{e^{\sqrt{k}}}^x \frac{1}{y \log^{2} y} dy   \bigg).
\end{align*}
As
$$
\int_{e^{\sqrt{k}}}^x \frac{dy}{y} = \log y \Big|_{t = e^{\sqrt{k}}}^{x} = (\log x) - \sqrt{k}
$$
and
$$
\int_{e^{\sqrt{k}}}^x \frac{1}{y \log^{2} y} dy = -\dfrac{1}{\log y} \Big|_{t = e^{\sqrt{k}}}^{x} 
= -\dfrac{1}{\log x} + \dfrac{1}{\sqrt{k}},
$$
we deduce
\begin{equation}
\label{eq:x5}
\int_{e^{\sqrt{k}}}^x \frac{\pi(y)(\log y - 1)}{y^2} dy=\log x+O({\sqrt{k}}).
\end{equation}
From \eqref{eq:x6} and \eqref{eq:T2}, we now obtain
$$
\sum_{\substack{e^{\sqrt{k}} <  p\le x\\ p\equiv \ell{\hskip -6pt}\pmod k}} \frac{\log p}{p} = \frac{\log x}{\phi(k)}+O\left(\frac{\sqrt{k}}{\phi(k)}\right).
$$

Collecting everything together, we get that 
$$
T_1 \le \frac{\log p_{\ell,k}}{p_{\ell,k}}+O\left(\frac{{\sqrt{k}}\log k}{\phi(k)}\right)
$$
and 
$$
T_2 = O\left(\frac{\sqrt{k}}{\phi(k)}\right),
$$
and now \eqref{eq:x7} gives the desired estimate. 
\end{proof}

\medskip
\noindent
{\bf Remark.} To get the estimate with ${\sqrt{k}}$ replaced by $k^{\delta}$ for any $\delta>0$ fixed, we may in the above calculation split the sum at $e^{k^{\delta}}$ and note that in the upper range $y\ge e^{k^{\delta}}$, we have $k\le (\log y)^{1/\delta}$ so we are still in the range of applicability of the Siegel-Walfitz theorem. We give no further details. 

\medskip 
\begin{proof}[Second proof of Lemma~\ref{lemmatwo} (with $m_{0} = 6$)]
Assume we have a solution to \eqref{lem2proof1eq1} with $m' = m-1 \ge 5$ and $n$ large, that is larger than some fixed amount that we can choose as we want.  The goal is to obtain a contradiction.  
We look at the sizes of $b$ and $m$ versus $n$.  We take $2$-adic valuations of both sides of \eqref{lem2proof1eq1}.  On the left--hand side, taking $n \ge 4$, we have, as in the proof of Lemma \ref{lem:1}, that
$$
\nu_2(n!)\ge n/2.
$$
For the right--hand side of \eqref{lem2proof1eq1}, we write $m' = 2^{r} t$, where $r$ and $t$ are nonnegative integers with $t$ odd.  Then
$$
\dfrac{b^{m'}-1}{b-1} = \dfrac{b^{2^{r}}-1}{b-1} \cdot \sum_{j=0}^{t-1} b^{2^{r}j}
= \prod_{i=0}^{r-1} (b^{2^{i}}+1) \cdot \sum_{j=0}^{t-1} b^{2^{r}j}.
$$
Since $t$ is odd, the sum on the right is odd.  If $b$ is even, then the product is odd as well.  If $b$ is odd, then $b^{2^{i}} + 1 \equiv 2 \pmod{4}$ for $i \ge 1$.  Noting that 
$$
\nu_{2}(m-1) = \nu_{2}(m') = r,
$$
we deduce that 
$$
\nu_2\left(b\left(\frac{b^{m-1}-1}{b-1}\right)\right)
\le \nu_2(b(b+1))+\nu_2(m-1) = \nu_2(b(b+1)) + O(\log n),
$$
where the error term follows for example from the bound on $m$ in Lemma~\ref{lem:1}.
We deduce that 
$$
\nu_2(b(b+1)) \ge \dfrac{n}{2} + O(\log n) \ge \dfrac{n}{3}, 
$$
where we have used that $n$ is sufficiently large in the last inequality.
In the above calculation, we may omit the factor $b+1$ when $m$ is even. 
We see now that $2^{\lfloor n/3\rfloor}$ divides one of $b$ or $b+1$ so that $\log b \gg n$. 
From \eqref{lem2proof1eq1}, we see that $n\log n \asymp m\log b$. 
Therefore, we obtain $m \ll \log n$. 

Setting $\delta = \gcd(m',2)-1 = \gcd(m-1,2)-1$, we rewrite \eqref{lem2proof1eq1} as
\begin{equation}
\label{eq:3}
n! = b(b+1)^{\delta}\left(\frac{b^{m-1}-1}{b^{1+\delta}-1}\right)
= b(b+1)^{\delta}\left(\frac{b^{m'}-1}{b^{1+\delta}-1}\right).
\end{equation}
Consider prime factors $q$ of $(b^{m'}-1)/(b^{1+\delta}-1)$.  Suppose first that $b \equiv \pm 1 \pmod{q}$.  
If $m'$ is even, then $\delta = 1$ and $(b^{m'}-1)/(b^{1+\delta}-1) = b^{m'-2} + b^{m'-4} + \cdots + b^{2} + 1$. 
In this case, we deduce $0 \equiv (b^{m'}-1)/(b^{1+\delta}-1) \equiv m'/2 \pmod{q}$.  If $m'$ is odd, then $\delta' = 0$ and $(b^{m'}-1)/(b^{1+\delta}-1) = b^{m'-1} + b^{m'-2} + \cdots + b + 1$.  Here, in the case when $b \equiv 1 \pmod{q}$, then we get $0 \equiv m' \pmod{q}$; and if $b \equiv -1 \pmod{q}$, then we have $0 \equiv 1 \pmod{q}$, an impossibility.  In any case, we deduce that in the case when $b \equiv \pm 1 \pmod{q}$, then $q$ divides $m'$.
If $q$ does not divide $b \pm 1$, then the minimal $d$ such that 
$q \mid (b^d-1)$ has the property that $d \mid (m-1)$ and $d \ge 3$. 
Furthermore, $d$ is the multiplicative order of $b$ modulo $q$, so $d\mid (q-1)$. 
In particular, $q \equiv 1 \pmod{d}$.  Summarizing, all the primes numbers $q \le n$ 
must appear in the right--hand side of \eqref{eq:3} and for these there are four possibilities:
\begin{itemize}
\item[(i)] $q\mid b$, and necessarily $q$ does not divide $(b+1)(b^{m'}-1)/(b^{1+\delta}-1)$;
\item[(ii)] $\delta=1$, $q\mid (b+1)$ but $q$ does not divide $m' b (b^{m'}-1)/(b^{1+\delta}-1)$;
\item[(iii)] $q\mid m'$;
\item[(iv)] $q$ divides $(b^{m'}-1)/(b^{1+\delta}-1)$, $q$ is relatively prime to $m' (b^{2}-1)$, and there exists a divisor $d \ge 3$ of $m'$ such that $q \equiv 1 \pmod d$.
\end{itemize}
Furthermore, we see that
\begin{equation}
\label{eq:3.5}
\nu_q(n!)=\left\lfloor \frac{n}{q}\right\rfloor+\left\lfloor \frac{n}{q^2}\right\rfloor+\cdots=\frac{n}{q}+O\left(1+\frac{n}{q^2}\right).
\end{equation}

By Stirling's formula, we have
\begin{equation}
\label{eq:3.5pt5}
\sum_{q \le n} \nu_{q}(n!) \log q = \log (n!) = n \log n + O(n).  
\end{equation}
We turn to estimating the left-hand side of \eqref{eq:3.5pt5} by looking at the contribution of each of (i)-(iv) to the logarithm of the right-hand side of \eqref{eq:3}; our plan is to then compare this to \eqref{eq:3.5pt5}.  To get started, we estimate the logarithms of both sides of the equation
$$
n! + 1 = \dfrac{b^{m}-1}{b-1}.
$$ 
By Stirling's formula for $n!$ again, we have
$$
\log(n!+1)=\log(n!)+O(1)=n\log n+O(n).
$$
On the other hand, we also see that
$$
\log ((b^{m}-1)/(b-1))=\log (b^{m-1})+O(1) = m' \log b+O(1).
$$
This shows that 
\begin{equation}
\label{eq:4.1}
\sum_{q \mid b} \nu_{q}(b) \log q = \log b = \frac{n\log n}{m'}+O\left(\frac{n}{m'}\right).
\end{equation}
Furthermore, 
\begin{equation}
\label{eq:4.2}
\sum_{q \mid (b+1)} \nu_{q}(b+1) \log q = \log(b+1)=\log b+O(1)=\frac{n\log n}{m'}+O\left(\frac{n}{m'}\right).
\end{equation}
This bounds the contribution of (i) and (ii) to the logarithm of the right-hand side of \eqref{eq:3}.
 
From \eqref{eq:3.5}, we obtain 
\begin{align*}
\sum_{q \mid m'} \nu_q(n!) \log q &= \sum_{q\mid m'} \left(\frac{n\log q}{q}+O\left( \log q+\frac{n\log q}{q^2}\right)\right)\\
&= n\sum_{q\mid m'} \frac{\log q}{q}+O\left(\sum_{q \le m'} \log q+n\sum_{q \le m'} \frac{\log q}{q^2}\right).
\end{align*}
Since $m' = m-1 \ll \log n$, the error term above is $O(n)$. 
As for the main term,  let 
$$
k=\omega(m')\ll \log m'/\log\log m',
$$ 
where $\omega(m')$ is the number of distinct prime divisors of $m'$ and
the upper bound arises from the case where $m'$ is the product of consecutive primes starting at $2$ (cf.~\cite[p.~355]{HW}).  We note that the weaker estimate $k \le \log_{2} m' \ll \log m'$ will suffice below. 
Let $q_1<\cdots <q_k$ be all the prime factors of $m'$. 
If $p_{j}$ denotes the $j$-th prime, then we deduce from $p_{k} \sim k \log k$ and $m' \le m \ll \log n$ that
\begin{align*}
\sum_{q\mid m'} \frac{\log q}{q} &= \sum_{i=1}^k \frac{\log q_i}{q_i} \le 
\sum_{i=1}^k \frac{\log p_i}{p_i} \\[5pt]
&\ll \log p_k \ll \log k \ll \log\log m' \ll \log\log\log n.
\end{align*}
Hence, we see that
\begin{equation}
\label{eq:4.3}
\sum_{q\mid m'} \nu_q(n!)\log q \ll n\log\log\log n,
\end{equation}
accounting for the contribution of (iii) to the logarithm of the right-hand side of \eqref{eq:3}.

We now turn to the contribution of (iv) to the logarithm of the right-hand side of \eqref{eq:3}.
For these, there exists $d\ge 3$ dividing $m'$ such that $q\equiv 1\pmod d$. 
We take the $\phi(m')$ classes $c \pmod {m'}$ with $1 \le c \le m'$ and $\gcd(c,m')=1$, and let 
${\mathcal C}$ denote the subset of these $c$ for which $c-1$ is divisible by some divisor of $m'$ (depending on $c$) that is $\ge 3$.  Observe that ${\mathcal C}$ does not include all $\phi(m')$ positive integers $c \le m'$ that are relatively prime to $m'$ since $\gcd((m'-1)-1,m')=\gcd(m'-2,m')\mid 2$, so $m'-1 \not\in {\mathcal C}$. Fix $c \in \mathcal C$.  
From \eqref{eq:x3} and \eqref{eq:x4}, we have 
$$
\sum_{\substack{q\le n\\ q\equiv c{\hskip -6pt}\pmod {m'}}} \frac{\log q}{q}=\frac{\log n}{\phi(m')}+A_{c,m'},
$$
where
$$
A_{c,m'}=O\left(\frac{\log q_{c,m'}}{q_{c,m'}}+\frac{{\sqrt{m'}}\log(m')}{\phi(m')}\right),
$$
with $q_{c,m'}$ denoting the smallest prime $q\equiv c\pmod {m'}$.  The term $\log q_{c,m'}/q_{c,m'}$ can be omitted when $q_{c,m'} \ge m'$ since then the second term $\sqrt{m'}\log(m')/\phi(m')$ is larger than $\log q_{c,m'}/q_{c,m'}$.
Summing over $c\in {\mathcal C}$, we obtain
\begin{equation}
\label{eq:7}
\sum_{c\in {\mathcal C}} \sum_{\substack{q\le n\\ q\equiv c{\hskip -6pt}\pmod {m'}}} \frac{\log q}{q}=\frac{|{\mathcal C}| \log n}{\phi(m')}+\sum_{c\in {\mathcal C}} A_{c,m'}.
\end{equation}
Recalling $m \ll \log n$, we see that
\begin{equation*}
\begin{split}
\sum_{c\in {\mathcal C}} A_{c,m'} &= O\left(\sum_{q\le m'} \frac{\log q}{q}+{\sqrt{m'}}\log (m')\right) \\[5pt]
&= O\big({\sqrt{m}}\log m\big)=O\big((\log n)^{1/2} \log\log n\big).
\end{split}
\end{equation*}
From \eqref{eq:3.5} and \eqref{eq:7}, we deduce now that
\begin{equation}
\label{eq:8}
\begin{split}
\sum_{c\in {\mathcal C}} \sum_{\substack{q\le n\\ q\equiv c{\hskip -6pt}\pmod {m'}}} \nu_{q}(n!) \log q
&\ = \ \sum_{c\in {\mathcal C}} \sum_{\substack{q\le n\\ q\equiv c{\hskip -6pt}\pmod {m'}}} \!\dfrac{n\,\log q}{q} \\
&\qquad \quad + O\bigg( \sum_{q\le n} \log q  \bigg)
+ O\bigg(  n  \sum_{q \le n}  \dfrac{\log q}{q^{2}} \bigg) \\[5pt]
&\ = \ \frac{|{\mathcal C}| n \log n}{\phi(m')} + O\big( n \,(\log n)^{1/2} \log\log n \big),
\end{split}
\end{equation}
which is an upper bound on the contribution of (iv) to the logarithm of the right-hand side of \eqref{eq:3}.

Collecting \eqref{eq:4.1}, \eqref{eq:4.2}, \eqref{eq:4.3} and \eqref{eq:8},  we thus get that the logarithm of the right--hand side of \eqref{eq:3} is bounded above by
$$
\left(\frac{1+\delta}{m'}+\frac{|{\mathcal C}|}{\phi(m')}\right) n\log n+O(n(\log n)^{1/2} \log\log n),
$$
where as before $\delta = \gcd(m',2)-1$.
By \eqref{eq:3} and \eqref{eq:3.5pt5}, we deduce
$$
\left(1-\left(\frac{1+\delta}{m'}+\frac{|{\mathcal C}|}{\phi(m')}\right)\right)n\log n \ll n(\log n)^{1/2} \log\log n
$$
so that
\begin{equation}
\label{eq:6}
1-\left(\frac{1+\delta}{m'}+\frac{|{\mathcal C}|}{\phi(m')}\right) \ll \frac{\log\log n}{(\log n)^{1/2}}.
\end{equation}
For a fixed $m'$, if the left-hand side of \eqref{eq:6} is positive, then \eqref{eq:6} cannot hold for $n$ large.  A direct computation shows that the left-hand side of \eqref{eq:6} is positive for $m' = m-1 = 3$ and
$5 \le m' = m-1 \le 1000$.  Thus, we may suppose now that $m' = m-1 > 1000$ (though a smaller lower bound will suffice for the arguments below).

We show now that the left--hand side of \eqref{eq:6} is positive for $m' > 1000$ and smaller than the right-hand side of \eqref{eq:6} for $n$ large (independent of $m'$) to obtain a contradiction.  
To see why, recall $|{\mathcal C}| \le \phi(m')-1$ since $m'-1 \not\in {\mathcal C}$.  
Writing $m' = 2^{s} m_{0}$ where $m_{0}$ is an odd integer, 
one can check directly that $a \not\in {\mathcal C}$ where
$$
a = 
\begin{cases}
2 &\text{if } s = 0, \\
m_{0}+2 &\text{if $s \ge 1$ and $m_{0} \equiv 1 {\hskip -6pt}\pmod{4}$ or if $s = 1$ and $m_{0} \equiv 3 {\hskip -6pt}\pmod{4}$},  \\
3m_{0} + 2 &\text{if $s \ge 2$ and $m_{0} \equiv 3 {\hskip -6pt}\pmod{4}$}.
\end{cases}
$$
Furthermore, $1 \le a < m'-1$ and $\gcd(a,m') = 1$.
We deduce then that 
$$
\frac{|{\mathcal C}|}{\phi(m')}\le \frac{\phi(m')-2}{\phi(m')} \le \frac{m'-3}{m'-1},
$$
where we have used that $\phi(m') \le m'-1$.
Since $\delta \in \{ 0,1 \}$, we obtain
$$
1-\left(\frac{1+\delta}{m'}+\frac{|{\mathcal C}|}{\phi(m')}\right)
\ge 1-\left(\frac{2}{m'}+\frac{m'-3}{m'-1}\right)
= \frac{2}{m'(m'-1)},
$$
and, with $n$ large, the inequality \eqref{eq:6} does not hold if $m' \le (\log n)^{1/5}$. 

Suppose next that $m' > (\log n)^{1/5}$. By sieve methods, with $n$ large, the interval $(m'/2, m']$ contains $\pi_1\gg m'/(\log m'\log\log m')$ primes $q$ such that $(q-1)/2$ is divisible only by primes $\ge (\log m')^3$. 
Such $q$ are relatively prime to $m'$ but not in ${\mathcal C}$ unless there is some prime $\ge (\log m')^3$ dividing both $m'$ and $q-1$.  For a fixed prime $r \ge (\log m')^3$ dividing $m'$, the number of primes $q$ counted in $\pi_{1}$ with $(q-1)/2$ divisible by $r$ is $\le m'/r \le m'/(\log m')^3$. Since $r \ge (\log m')^3$ and $r \mid m'$, there are $\le \log m'/\log\log m'$ possibilities for $r$.  Hence, the number of primes $q$ counted by $\pi_{1}$ which are in ${\mathcal C}$ is 
$$
\le \dfrac{m'}{(\log m')^3} \cdot \dfrac{\log m'}{\log\log m'} = \dfrac{m'}{(\log m')^2 \log\log m'} < \dfrac{\pi_1}{2}
$$
for $n$, and hence $m'$, large.  
We deduce now that $|{\mathcal C}| \le \phi(m')-\pi_1/2$.  Thus, 
$$
\frac{2}{m'}+\frac{|{\mathcal C}|}{\phi(m')}\le \frac{2}{m'}+\frac{m'-\pi_1/2}{m'}=\frac{m'-(\pi_1/2-2)}{m'}.
$$
Recalling the upper bound $m' = m-1 \ll \log n$, we obtain
$$
1-\left(\frac{2}{m'}+\frac{|{\mathcal C}|}{\phi(m')}\right)
\ge \frac{\pi_1/2-2}{m'} \gg \frac{1}{\log m'\log\log m'}
\gg \frac{1}{\log\log n \log\log\log n}.
$$
For large $n$, we deduce that \eqref{eq:6} does not hold, giving the desired contradiction.  \end{proof}

We now turn to the proof of Theorem~\ref{thm:1}.

\begin{lemma}\label{thm1:lemma}
Let $p$ be a prime, and let $a \in \mathbb Z$ with $1 \le |a| \le n$.  Set
$$
P_{p,a}(X) = X^{p-1} + \cdots + X + 1 + a = \frac{X^p-1}{X-1}+a.
$$
If an integer $b \ge 2$ is such that
\begin{equation}
\label{eq:y}
n!=P_{p,a}(b),
\end{equation}
then $a$ satisfies one of the following:
\begin{itemize}
\item[(i)] $|a| \equiv 1 \pmod p$;
\item[(ii)] $a = \pm \,p a_1$, where $a_{1} \in \mathbb Z^{+}$ with $a_1 \equiv 1\pmod p$.
\end{itemize}
\end{lemma}

\begin{proof}
The condition $1 \le |a| \le n$ implies that $a \mid n!$.  As a consequence of \eqref{eq:y}, we see that
$$
a \mid (b^p-1)/(b-1).
$$
Note that $(X^p-1)/(X-1) = X^{p-1} + \cdots + X + 1$ is the $p$-th cyclotomic polynomial $\Phi_{p}(X)$.  
As in the first proof of Lemma~\ref{lemmatwo}, where we considered the cyclotomic polynomial $\Phi_{m'}(X)$, 
we have $\Phi_{p}(b) = (b^p-1)/(b-1)$ is divisible only by primes congruent to $1$ modulo $p$ or possibly by $p$ itself, but not by $p^2$.  
Since $a \mid \Phi_{p}(b)$, we deduce that $a$ satisfies (i) or (ii), completing the proof.
\end{proof}

\begin{proof}[Proof of Theorem~\ref{thm:1}]
We work on equation \eqref{eq:x} with a prime $p\ge 7$.  First, we assume $a$ satisfies
\begin{equation}
\label{thm1proofeq1}
1 \le a < \dfrac{0.5 \log\log n}{\log \log \log n},
\end{equation}
with the goal of obtaining a contradiction.
As in the statement of Theorem \ref{thm:1}, we take $n$ large.  
Observe that \eqref{eq:x} is equivalent to \eqref{eq:y}; we will work with the latter below. 
By Lemma~\ref{thm1:lemma}, we see that either (i) or (ii) holds.  

Assume that (i) holds, so $a \equiv 1\pmod p$.  We also have
$$
(b^p-1)/(b-1) = b^{p-1} + \cdots + b + 1 \equiv 
\begin{cases}
0 {\hskip -5pt}\pmod p &\text{if } b \equiv 1 {\hskip -5pt}\pmod{p}, \\
1 {\hskip -5pt}\pmod p &\text{otherwise}.
\end{cases}
$$
Therefore, we obtain $P_{p,a}(b) \equiv 1 \text{ or } 2 \pmod{p}$.  
Since $p=m \ll n^{1/2}(\log n)^{2}$ by Lemma \ref{lem:1} and since $n$ is large, we get that $p\mid n!$.  
Thus, the left-hand side of \eqref{eq:y} is $0$ modulo $p$ and the right-hand side is $1$ or $2$ modulo $p$, contradicting $p \ge 7$.

We deduce that (ii) holds.
Since $a$, $p$ and $a_{1}$ are positive, the sign in (ii) is positive; in other words, $a = p a_1$. 
Making the substitution $X=Y+1$, we get
\begin{align*}
P_{p,a}(Y+1) &= \frac{(Y+1)^p-1}{Y}+pa_1 \\
&= Y^{p-1}+\binom{p}{1} Y^{p-2}+\cdots+\binom{p}{p-2}Y+p(1+a_1).
\end{align*}
Since $P_{p,a}(Y+1)$ is Eisenstein with respect to $p$, we obtain that $P_{p,a}(Y+1)$ and consequently $P_{p,a}(X)$ are irreducible.  
To rule out the existence of solutions to \eqref{eq:y} for large $n$, it suffices now to show that the smallest prime $q$ for which $P_{p,a}(X) \equiv 0 \pmod{q}$ has no solutions is $\le n$, as then the left-hand side of \eqref{eq:y} will be divisible by $q$ and the right-hand side will not be, giving us the desired contradiction.
The idea is motivated by the fact (cf.~\cite[Theorems~1, 2]{Serre}) that the set of primes $q$ for which $P_{p,a}(X)\equiv 0\pmod q$ has no solutions is of positive density and in fact of density $\ge 1/\deg(P_{p,a}(X)) = 1/(p-1)$.  For the smallest such $q$, we use a theorem of J.~C.~Lagarias, H.~L.~Montgomery and A.~M.~Odlyzko \cite[Theorem~1.1]{LMO} implying a prime $q$ exists for which $P_{p,a}(X) \equiv 0 \pmod{q}$ has no solutions and
$$
q \le 2 D^A,
$$
where $A$  is an absolute constant, $D=|\Delta|^{p!}$ and $\Delta$ is the discriminant of $P_{p,a}(X)$.
Since 
\begin{equation}\label{xminus1timesP}
(X-1)P_{p,a}(X)=X^{p}+aX-(a+1),
\end{equation}
and $P_{a,b}(1) = a+p$, we get that 
\begin{equation}\label{Deltavalue}
|\Delta|=\frac{p^{p} (a+1)^{p-1}+a^p(p-1)^{p-1}}{(a+p)^2}
\end{equation}
(see \cite[Theorem~2]{Swan}).  Recall $p \ge 7$ and that we are in the case (ii) where $p \mid a$.  Therefore,
$$
|\Delta|\le (a+1)^{2p},
$$
and we deduce
$
D \le (a+1)^{2 \cdot (p+1)!}.
$
Since $p \le a$ and $(p+1)! \le (p+1)^{p} < (p+1)^{p+1}/2$ for $p \ge 7$, we obtain
$$
D \le (a+1)^{(a+1)^{a+1}}.
$$
As we now have
$$
q \le 2 (a+1)^{A (a+1)^{a+1}},
$$
we obtain a contradiction if $n$ is equal to or larger than this upper bound on $q$.  Therefore, 
$$
n <  2 (a+1)^{A (a+1)^{a+1}}.
$$
As $n$ is large, we now must have $a$ is large, and
\begin{align*}
\log\log n &\le \log\big( A (a+1)^{a+1} \log(a+1) \big) + O(1) \\
&= (a+1) \log (a+1) + O(\log\log (a+1)).
\end{align*}
Given $n$ is large, we obtain a contradiction to our assumption \eqref{thm1proofeq1}.  Thus, \eqref{thm1eq} holds, finishing the proof.
\end{proof}

\begin{proof}[Proof of Theorem~\ref{thm:2}]
(i)
Recall that before the statement of Theorem~\ref{thm:1}, we established that \eqref{eq:x} has no solutions if $a = 0$.  Theorem~\ref{thm:2}~(i) follows directly now from Lemma~\ref{thm1:lemma}.
\vskip 5pt \noindent
(ii)
In the case that $a > 0$, the result follows from Theorem~\ref{thm:1}.  We also saw in Lemma~\ref{lemmatwo},  with the second proof where $m_{0} = 6$, that \eqref{eq:x} has no solutions when $n$ is sufficiently large and $a = -1$.  So we restrict now to $a < -1$.  Observe that Theorem~\ref{thm:2} (i) established above implies $p \le |a|$.  Since $p \ge 7$, we deduce $|a| \ge 7$.  

We will want a bound on the discriminant $D$ of the splitting field for $P_{p,a}(X)$ over $\mathbb Q$.  We can make use of \eqref{Deltavalue} as before, where $\Delta$ is the discriminant of $P_{p,a}(X)$, except there is a possibility that the denominator $(a+p)^{2}$ is $0$.  This happens precisely when $a = -p$ and $P_{p,a}(1) = 0$.  We estimate $\Delta$ differently in a way that can be used more generally here.  Recall that $p \ge 7$.  From \eqref{xminus1timesP}, the roots of $P_{p,a}(X)$ have absolute value no greater than the positive real root of
$$
X^{p} - |a| X - |a+1|,
$$
which one can check is no more than $|a|^{2/p}$; indeed, this follows from $|a+1| < |a|$ and
$$
|a| \big( |a| - |a|^{2/7} - 1 \big) > 0 \qquad \text{for } |a| \ge 3.
$$
Thus, the difference of any two distinct roots of $P_{p,a}(X)$ is at most $2 |a|^{2/p}$.  As $P_{p,a}(X)$ is a monic polynomial of degree $p-1$, we deduce that 
$$
|\Delta| \le (2|a|^{2/p})^{(p-1)(p-2)} \le 2^{p^{2}} |a|^{2p}.
$$
Recalling $|a| \ge 7$, we see that $2 \le |a|^{1/2}$ so that
$$
|\Delta| \le |a|^{p^{2}/2 + 2p} \le |a|^{p^{2}}.
$$
Therefore, 
$$
|D| \le |\Delta|^{(p-1)!} \le |a|^{p \cdot p!}.
$$
As $p!$ is the product of the first $p-1$ positive integers $> 1$, we have $p! \le p^{p-1}$.  Since $p \le |a|$, we deduce that 
$$
|D| \le |a|^{p^{p}} \le |a|^{|a|^{|a|}}.
$$

If $\Phi_{p}(x) + a$ is irreducible, then $P_{p,a}(X)$, from Lemma~\ref{thm1:lemma}, is irreducible.  The conclusion of Theorem~\ref{thm:2} (ii) follows from an argument identical to the proof of Theorem~\ref{thm:1} after it was shown there that $P_{p,a}(X)$ is irreducible.  More precisely, with the bound on $|D|$ above, 
a prime $q$ exists for which $P_{p,a}(X) \equiv 0 \pmod{q}$ has no solutions and
$$
q \le 2 |a|^{A\,|a|^{|a|}}.
$$
As in the proof of Theorem~\ref{thm:1}, we deduce then that the smallest prime $q$ such that $P_{p,a}(X) \equiv 0 \pmod{q}$ has no solutions is $\le n$, provided that $n$ is large (not dependent on $p$ and $b$) and \eqref{thm1eq}, with $a$ replaced by $|a|$, does not hold.  As $q$ will necessarily divide the left-hand side of \eqref{eq:x}, or equivalently \eqref{eq:y}, in this case but not divide the right-hand side, we deduce that, with $n$ sufficiently large, we have \eqref{thm1eq} with $a$ replaced by $|a|$.  Thus, also \eqref{thm2eq} holds.

We now address what happens if $\Phi_{p}(X) + a$, equivalently $P_{p,a}(X)$, factors as a linear polynomial times an irreducible polynomial in $\mathbb Z[X]$, where we call the latter $Q(X)$.  
Since $P_{p,a}(X)$ is monic, we also can and do take $Q(X)$ to be monic.  Note that $\deg Q = p-2$.  
Let ${\mathcal S}$ be the set of primes $q$ such that $Q(X)\equiv 0\pmod q$ has a solution.  Jordan's theorem  (cf.~\cite[Theorems~1, 2]{Serre}) tells us that 
$$
d{\mathcal S} := \lim_{T \rightarrow \infty} \dfrac{|\{ q \in \mathcal S: q \le T \}|}{\pi(T)} \le 1-\dfrac{1}{p-2} = \dfrac{p-3}{p-2}.
$$ 
Observe that ${\text{\rm deg}}(Q)/{\text{\rm deg}}(P)=(p-2)/(p-1)>d{\mathcal S}$.
Theorem~4.1 in \cite{BH} implies for fixed $a$ and $p$ that equation \eqref{eq:x} has only finitely many solutions $(n,b,p)$.  We justify next that with a slight modification of the argument for Theorem~4.1 in \cite{BH}, using an effective version of the Chebotarev Density Theorem (cf.~\cite{ThornerZaman}), we can obtain the same result uniformly for $a$ instead satisfying 
\begin{equation}\label{boundforminusa}
|a| < \frac{0.5\log\log\log n}{\log\log\log\log n}.
\end{equation}

Assume now that \eqref{boundforminusa} holds (with $n$ sufficiently large and \eqref{eq:x} holding).  
Since $P_{p,a}(X)$ and $Q(X)$ are monic, the linear factor of $P_{p,a}(X)$ is of the form $X - X_{0}$ where $X_{0} \in \mathbb Z$ and $|X_{0}|$ divides $|a+1|$ where we recall that $a < -1$.  The latter implies $|X_{0}| \le |a|$.  We use $\varepsilon_{j} > 0$ below to denote arbitrarily chosen sufficiently small numbers, independent of $n$, $p$, $a$ and $b$, but note that the condition that $n$ is sufficiently large in Theorem~\ref{thm:2}~(ii) may depend on the choices of $\varepsilon_{j}$.  
Since $n$ is sufficiently large, Lemma~\ref{lem:1} implies that $b$ is also as large as we need.  In particular, the lower bound on $b$ in Lemma~\ref{lem:1} implies that, uniformly for $a$ satisfying \eqref{boundforminusa}, we have
$$
P_{p,a}(b) \le (1+\varepsilon_{1}) b^{p-1} 
\qquad \text{and} \qquad
Q(b) = \dfrac{P_{p,a}(b)}{b - X_{0}} \ge (1-\varepsilon_{1}) b^{p-2}.
$$
Recall $P_{p,a}(b) = n!$.  
We deduce that
\begin{equation}\label{Q(b)ineq1}
Q(b) \ge (1- \varepsilon_{2}) P_{p,a}(b)^{(p-2)/(p-1)} = (1- \varepsilon_{2}) (n!)^{(p-2)/(p-1)},
\end{equation}
where we can take $\varepsilon_{2} = 2\varepsilon_{1}$ since then
$$
1- \varepsilon_{2} = 1- 2\varepsilon_{1}
\le \dfrac{1-\varepsilon_{1}}{1+\varepsilon_{1}} 
\le \dfrac{1-\varepsilon_{1}}{(1+\varepsilon_{1})^{(p-2)/(p-1)}}.
$$
Note that $Q(b) > 0$.
Recalling the definition of $\mathcal S$, we obtain from \eqref{eq:y} that 
$$
Q(b) \ \text{ divides }\ \prod_{q \in \mathcal S} q^{\nu_{q}(n!)},
$$
which we now use to find an upper bound on $Q(b)$.  
Recall that $\Delta$ is the discriminant of $P_{p,a}(X)$.  Let $\Delta'$ denote the discriminant of $Q(X)$.  As $P_{p,a}(X) = Q(X) (X-X_{0})$, we deduce that $|\Delta| = |\Delta'| |Q(X_{0})|^{2} \ge |\Delta'|$ (where we have used $Q(X)$ is irreducible and of degree $p-2 > 1$ so that $Q(X_{0}) \ne 0$).  
Taking $D'$ to be the discriminant of the splitting field for $Q(X)$ over $\mathbb Q$, we see that our arguments bounding $D$ above give us
$$
|D'| \le |\Delta'|^{(p-2)!} \le |\Delta|^{(p-1)!} \le |a|^{|a|^{|a|}}.
$$

From \cite[Theorem~2]{Serre}, since $\deg Q = p-2$, we have 
$$
d \mathcal S \le 1 - \frac{1}{p-2} = \frac{p-3}{p-2}.
$$
As described in the introduction of \cite{ThornerZaman}, an effective form of the Chebotarev Density Theorem implies further that 
\begin{equation}\label{Sbound}
\begin{split}
(1-\varepsilon_{3}) \cdot d \mathcal S \cdot \dfrac{t}{\log t}
 &\le |\{ q \le t: q \in \mathcal S \}| \\
 &\le (1+\varepsilon_{3}) \cdot d \mathcal S \cdot \dfrac{t}{\log t}
 \quad \text{for all $t \ge \exp\bigg(\dfrac{\log n}{\log\log n}\bigg)$}
 \end{split}
\end{equation}
if $n$ is sufficiently large and
\begin{equation}\label{nD'ineq}
\dfrac{\log n}{\log\log n} \gg (\log |D'|)^{2} + p! (\log (p!))^{2} + p!^{p!} \log |D'| + |D'|.
\end{equation}
Recalling $7 \le p \le |a|$ and the bound on $|D'|$, we see that \eqref{nD'ineq} holds provided
$$
\dfrac{\log n}{\log\log n} \gg |a|^{2|a|} \log^{2} |a| + |a|^{|a|} (|a| \log |a|)^{2} + |a|^{|a|^{|a|}} |a|^{|a|} \log |a| + |a|^{|a|^{|a|}}.
$$
Given that $n$ is sufficiently large, this last inequality holds whenever \eqref{boundforminusa} holds.  
Thus, \eqref{boundforminusa} implies \eqref{Sbound}.  

Define
$$
\mathcal S' = \bigg\{ q \in \mathcal S : q > \exp\bigg(\dfrac{\log n}{\log\log n}\bigg) \bigg\}.
$$
For a prime $q$, we have
$$
\nu_{q}(n!) \le \sum_{j=1}^{\infty} \bigg\lfloor \dfrac{n}{q^{j}} \bigg\rfloor
< \sum_{j=1}^{\infty} \dfrac{n}{q^{j}} = \dfrac{n}{q-1}.
$$
Thus,
$$
\prod_{q \in \mathcal S} q^{\nu_{q}(n!)}
= \prod_{\substack{q \in \mathcal S \\ q \le n}} q^{\nu_{q}(n!)}
\le \prod_{\substack{q \in \mathcal S \\ q \le n}} q^{n/(q-1)}.
$$
Set $J = \lceil (\log n)/\log(1+\varepsilon_{4}) \rceil$, where $\varepsilon_{4} = \sqrt{\varepsilon_{3}}$.  
Define the interval 
$$
I_{j} = \big( (1+\varepsilon_{4})^{j-1}, (1+\varepsilon_{4})^{j} \big],
\quad \text{ for } 1 \le j \le J.
$$  
Fix $j_{0} > 1$ such that 
$$
(1+\varepsilon_{4})^{j_{0}-1} < 2\exp\bigg(\dfrac{\log n}{\log\log n}\bigg) \le (1+\varepsilon_{4})^{j_{0}},
$$
and observe that for $j \ge j_{0}$ and $n$ sufficiently large, we have
\begin{equation}\label{ep6def}
\dfrac{(1+\varepsilon_{4})^{j-1}}{(1+\varepsilon_{4})^{j-1} - 1} < 1 + \varepsilon_{5},
\end{equation}
for some $\varepsilon_{5} > 0$.
For some constant $C > 0$, we deduce that
\begin{align*}
\log \prod_{q \in \mathcal S} q^{\nu_{q}(n!)}
&\le n \sum_{q \in \mathcal S} \dfrac{\log q}{q - 1} \\
&\le n \Bigg( \sum_{\substack{q \le 2\exp((\log n)/\log\log n) \\ q \text{ prime}}} \dfrac{\log q}{q - 1} + \sum_{j=j_{0}}^{J} \sum_{\substack{q \in \mathcal S' \\ q \in I_{j}}} \dfrac{\log q}{q - 1}  \Bigg) \\
&\le n \bigg( \dfrac{\log n}{\log\log n} + C + \,\sum_{j=j_{0}}^{J} 
|\{ q \in I_{j} \cap \mathcal S' \}| \cdot  \max\bigg\{ \dfrac{\log q}{q-1} :  q \in I_{j} \bigg\} \bigg).
\end{align*}
By \eqref{Sbound}, for $j \ge j_{0}$ so that $j/(j-1) \le 2$, we obtain that $|\{ q \in I_{j} \cap \mathcal S' \}|$ is bounded above by 
\begin{align*}
(1+\varepsilon_{3}) &\cdot d \mathcal S \cdot \dfrac{(1+\varepsilon_{4})^{j}}{j \log (1+\varepsilon_{4})}
- (1-\varepsilon_{3}) \cdot d \mathcal S \cdot \dfrac{(1+\varepsilon_{4})^{j-1}}{(j-1)\log (1+\varepsilon_{4})} \\[5pt]
&\le \bigg(  \dfrac{1+\varepsilon_{4}}{j}  - \dfrac{1}{j-1}\bigg) \cdot d \mathcal S \cdot \dfrac{(1+\varepsilon_{4})^{j-1}}{\log (1+\varepsilon_{4})} + 
3\varepsilon_{3} \cdot d \mathcal S \cdot \dfrac{(1+\varepsilon_{4})^{j}}{j \log (1+\varepsilon_{4})} \\[5pt]
&\le \dfrac{\varepsilon_{4}}{j} \cdot d \mathcal S \cdot \dfrac{(1+\varepsilon_{4})^{j-1}}{\log (1+\varepsilon_{4})} + 
3\varepsilon_{3} \cdot d \mathcal S \cdot \dfrac{(1+\varepsilon_{4})^{j}}{j \log (1+\varepsilon_{4})}.
\end{align*}
Recall \eqref{ep6def}.  
As $(\log x)/(x-1)$ is a decreasing function for $x > 1$, we deduce, for $j \ge j_{0}$, that
\begin{align*}
\max\bigg\{ \dfrac{\log q}{q-1} :  q \in I_{j} \bigg\} 
&\le \dfrac{(j-1) \log (1+\varepsilon_{4})}{(1+\varepsilon_{4})^{j-1}}
\cdot \dfrac{{(1+\varepsilon_{4})^{j-1}}}{{(1+\varepsilon_{4})^{j-1}}-1} \\[5pt]
&< (1+\varepsilon_{5}) \cdot \dfrac{j \log (1+\varepsilon_{4})}{(1+\varepsilon_{4})^{j-1}}.
\end{align*}
Therefore, for $j \ge j_{0}$, we see that
$$
|\{ q \in I_{j} \cap \mathcal S' \}| \cdot  \max\bigg\{ \dfrac{\log q}{q-1} :  q \in I_{j} \bigg\} 
\le \big(\varepsilon_{4} + 3 \varepsilon_{3} (1+\varepsilon_{4})\big) (1+\varepsilon_{5}) \cdot d \mathcal S.
$$
Recall $\varepsilon_{4} = \sqrt{\varepsilon_{3}}$.  Also, we have $\log(1+\varepsilon_{4}) \ge \varepsilon_{4} - \varepsilon_{4}^{2}/2$.  
From the definition of $J$ and noting $j_{0} \ge 2$, we obtain
\begin{align*}
\sum_{j=j_{0}}^{J} 
|\{ q \in I_{j} \cap \mathcal S' \}| &\cdot  \max\bigg\{ \dfrac{\log q}{q-1} :  q \in I_{j} \bigg\} \\[5pt]
&\le \big(\varepsilon_{4} + 3 \varepsilon_{3} (1+\varepsilon_{4})\big) (1+\varepsilon_{5}) \cdot d \mathcal S \cdot \dfrac{\log n}{\log(1+\varepsilon_{4})} \\[5pt]
&\le \big(\varepsilon_{4} + 3 \varepsilon_{4}^{2} (1+\varepsilon_{4})\big) (1+\varepsilon_{5}) \cdot d \mathcal S \cdot \dfrac{\log n}{\varepsilon_{4} - \varepsilon_{4}^{2}/2} \\[5pt]
&\le (1 + \varepsilon_{6}) \cdot d \mathcal S \cdot \log n.
\end{align*}
Recall $d \mathcal S \le (p-3)/(p-2)$.  
We deduce that
$$
\log \prod_{q \in \mathcal S} q^{\nu_{q}(n!)}
\le (1 + \varepsilon_{7}) \cdot d \mathcal S \cdot n \log n
\le (1 + \varepsilon_{7}) \cdot \dfrac{p-3}{p-2} \cdot n \log n,
$$
which is therefore also an upper bound on $\log Q(b)$.  
On the other hand, we see from \eqref{Q(b)ineq1} that
$$
\log Q(b) \ge \log (1-\varepsilon_{2}) + \dfrac{p-2}{p-1} \cdot \log(n!)
\ge (1-\varepsilon_{8}) \cdot \dfrac{p-2}{p-1} \cdot n \log n.
$$
As
$$
\dfrac{p-3}{p-2} < \dfrac{p-2}{p-1},
$$
we obtain a contradiction for $\varepsilon_{j}$ chosen sufficiently small.  
Therefore, \eqref{boundforminusa} does not hold and \eqref{thm2eq} follows.

\vskip 5pt \noindent
(iii)
Fix $n \in \mathbb Z^{+}$ and $a < 0$ such that $p \mid a$.  
Let $a_{1}$ be the negative integer for which $a = p a_{1}$.
By Theorem~\ref{thm:2} (ii), it suffices to show that $\Phi_{p}(X) + a$ is either irreducible or a linear polynomial times an irreducible polynomial.  It further suffices instead to establish the same result for 
$$
g(X) = \Phi_{p}(X+1) + a = X^{p-1} + \binom{p}{1} X^{p-2} + \cdots + \binom{p}{p-2} X + p (a_{1}+1).
$$
If $p \nmid (a_{1}+1)$, then $g(X)$ is Eisenstein and, hence, irreducible.
If $p \mid (a_{1}+1)$, then we make use of the Newton polygon of $g(X)$ with respect to $p$ (cf.~\cite[Section~2]{Fil}).  This Newton polygon consists of two line segments, one joining $(0,0)$ to $(p-2,1)$ and the other joining $(p-2,1)$ to $(p-1,e)$ where $e$ is an integer $\ge 2$.  By a theorem of G.~Dumas \cite{Dumas}, it follows that $g(X)$ is either an irreducible polynomial or a linear polynomial times an irreducible polynomial, giving the desired result.

\vskip 5pt \noindent
(iv)
As a condition in Theorem~\ref{thm:2}, we have $p \ge 7$.  
For $2 \le |a| \le n$ satisfying \eqref{eq:x}, we see that $|a| \mid \Phi_{p}(b)$ since $|a| \mid n!$.  
Also, by Theorem~\ref{thm:2} (i), we obtain that either $p \mid a$ or $p \mid (|a|-1)$. 
Therefore, we deduce $|a| \ge p \ge 7$.  
Since $p$ is odd, the value of $(b^{p}-1)/(b-1) = b^{p-1} + \cdots + b + 1$ is odd independent of the parity of $b$, so $|a| \mid \Phi_{p}(b)$ implies that $a$ is odd.  
Now, suppose $a$ has a prime factor $q \ne p$.  
Since \eqref{eq:x} implies $q \mid \Phi_{p}(b)$, we deduce $q \equiv 1 \pmod{p}$, as the order of $b$ will be $p$ modulo $q$ (cf.~\cite{Dic}).  

\vskip 5pt \noindent
(v) 
By Theorem~\ref{thm:2} (iv), we need only consider $a$ odd and $|a| > 1$. 
From Theorem~\ref{thm:2} (i), if \eqref{eq:x} holds and $(a,p) \ne (-p,p)$, then $p$ belongs to the set
$$
\mathcal P = \{ p \text{ prime} : p \mid (|a|-1) \text{ or } p \mid a, p \ge 7,  p \ne -a \}.
$$
For each $a$ odd with $1 < |a| \le 10^{5}$, we considered $p$ from the set $\mathcal P$.
We then looked at the primes $q \le \min\{ |a|, 100 \}$ in turn. 
Note that $\Phi_{p}(1) + a = p+a$ and, for $k \ne 1$, we have $\Phi_{p}(k) + a = (k^{p}-1)/(k-1) + a$.   
We checked if $p+a$ or some value of $(k^{p}-1)/(k-1) + a$ for $k \in \{ 0, 2, 3, \ldots, q-1 \}$ is $0$ modulo $q$.  If all of these were non-zero for some $q$, then $q \le |a| \le n$ so that $q$ divides the left-hand side of \eqref{eq:x} but not the right-hand side \eqref{eq:x}, so that \eqref{eq:x} does not hold.  
So when this occurred, we stopped considering $q \le \min\{ |a|, 100 \}$ and proceeded to the next value of $p \in \mathcal P$ or, if there were no more such $p$, to the next value of $a$.  This procedure was sufficient to obtain the list given in Theorem~\ref{thm:2} (v) and took less than a minute and a half to compute using Maple 2023 on a Macbook Pro with an Apple M2 chip.  

\vskip 5pt \noindent
(vi) 
For the case that $a < 0$ and $p \mid a$, in the proof of Theorem~\ref{thm:2} (iii), we showed that the polynomial $g(X) = \Phi_{p}(X+1)+a$ is either irreducible or a linear polynomial times an irreducible polynomial.  When $a = -p$, we have $\Phi_{p}(1)+a = 0$.  We deduce then that $\Phi_{p}(X)-p$ is $X-1$ times an irreducible polynomial.  Computations revealed also that for every choice of $(a,p) \in S_{0}$ not of the form $(-p,p)$, we have $\Phi_{p}(X)+a$ is a linear polynomial times an irreducible polynomial.  By Theorem~\ref{thm:2} (ii), the result follows.
\end{proof}
 
\medskip
\noindent
{\bf Remark.}
It is not hard to make $P_{p,a}(X)$ have a root modulo every integer $q$.  We can just choose $a$ so that $P_{p,a}(X)$ has a linear factor.  More precisely, choosing $a=-(b_0^{p}-1)/(b_{0}-1)$ for some integer $b_0 >  1$,  we get that $P_{p,a}(X)$ has the factor $X-b_0$. Is the co-factor $P_{p,a}(X)/(X-b_0)$ irreducible so that we can apply Theorem~\ref{thm:2} (ii)?  If this were true and provable, then perhaps one could restrict oneself 
to the case when $P_{p,a}(X)$ does not have a linear factor and in this case the conjecture following Theorem~3 in \cite{Sch2} would predict that under this instance there are only finitely many pairs $(a,p)$ with $a$ negative such that $P_{p,a}(X)$ does not have a linear factor and is not irreducible.  We leave such questions as future projects. 

\section{Acknowledgements}

The authors thank Carl Pomerance for useful correspondence. This work was initiated while the second author was a Fellow at the Stellenbosch Institute for Advanced Study in the second part of 2023. He thanks this Institution for hospitality, support and great working conditions.

\end{document}